\documentclass{mpaper}

\include{thref}
\include{symbols}

\usepackage{todonotes}

\title{Chains and antichains in the Weihrauch lattice}

\author[Lempp]{Steffen Lempp}
\author[Marcone]{Alberto Marcone}
\author[Valenti]{Manlio Valenti}

\address[Steffen Lempp]{Department of Mathematics\\
University of Wisconsin - Madison\\
Madison, Wisconsin 53706\\
USA}
\email{\href{mailto:lempp@math.wisc.edu}{lempp@math.wisc.edu}}

\address[Alberto Marcone]%
  {Dipartimento di Scienze Matematiche, Informatiche e Fisiche\\
  Universit\`a di Udine\\
  33100 Udine\\
  Italy}
\email{\href{mailto:alberto.marcone@uniud.it}{alberto.marcone@uniud.it}}

\address[Manlio Valenti]{Department of Computer Science\\
Swansea University\\
Swan\-sea, SA1 8EN\\
UK}
\email{\href{mailto:manliovalenti@gmail.com}{manliovalenti@gmail.com}}

\thanks{The first author's research was partially supported by AMS-Simons Foundation Collaboration Grant 626304.
The second and third authors' research was partially supported by the Italian PRIN 2017 \emph{Mathematical Logic: models, sets, computability}. The second author was also partially supported by the Italian PRIN 2022 \emph{Models, sets and classifications}, prot.\ 2022TECZJA, funded by the European Union - Next Generation EU. A significant part of this work was carried out while the third author was affiliated with the University of Udine and the University of Wisconsin-Madison. The authors would like to thank Paul Shafer for useful conversation on the topics of the paper. We also thank the anonymous referee for their careful reading of the paper and their valuable comments.}

\date{}

\subjclass{Primary 03D30; Secondary 03D78}
\keywords{Weihrauch degrees, Medvedev degrees, chains, cofinality, antichains}

\begin{document}

\maketitle

\begin{abstract}
    We study the existence and the distribution of ``long'' chains in the
    Weihrauch degrees, mostly focusing on chains with uncountable
    cofinality. We characterize when such chains have an upper bound and
    prove that there are no cofinal chains (of any order type) in the
    Weihrauch degrees. Furthermore, we show that the existence of coinitial
    sequences of non-zero degrees is equivalent to $\CH$. Finally, we
    explore the extendibility of antichains, providing some necessary
    conditions for maximality.
\end{abstract}

\section{Introduction}

Weihrauch reducibility is a notion of reducibility between partial
multi-valued functions that is useful to calibrate their uniform
computational strength. The theory of Weihrauch reducibility has often been
studied in connection with reverse mathematics, as it can be used to analyze
the computability-theoretical content of $\forall\exists$-statements.

Many results in the literature on Weihrauch degrees focus on characterizing
the degree of specific problems. In contrast, several natural questions on
the structure of Weihrauch degrees are still open. In this paper, we study
the existence of chains and antichains in the Weihrauch degrees.

In a given partial order $(P,\le)$, a \textdef{chain} is a linearly ordered
subset of $P$. Conversely, an \textdef{antichain} is a set of pairwise
incomparable elements of $P$. 
We use $\family{a_i}{i\in I}$ to denote a family indexed with elements in some set $I$. We stress that, for example, a chain $\family{a_i}{i\in \omega}$ of elements of $P$ could be ill-founded or even dense. On the other hand, if $L$ is a linear order, we write $\sequence{a_x}{x \in L}$ if the chain $\family{a_x}{x\in L}$ is order-isomorphic to $L$ via the map $x \mapsto a_x$.

A set $S\subseteq P$ is called \textdef{cofinal} (in $P$) if for every $p\in
P$ there is $q \in S$ such that $p\le q$. The \textdef{cofinality} of $P$,
denoted $\cofinality(P)$, is the least cardinality of a cofinal chain. A cardinal $\kappa$ is called \textdef{regular} if it is equal to its cofinality. The
dual notion of cofinality is \textdef{coinitiality}: a set $S\subseteq P$ is
called \textdef{coinitial} if for every $p\in P$ there is $q \in S$ such that
$q\le p$. Equivalently, a coinitial set in $P$ is a cofinal set in
$P^*:=(P,\ge)$. The \textdef{coinitiality} of $P$, denoted
$\coinitiality(P)$, is the least cardinality of a coinitial chain. Both
cofinality and coinitiality need not be well-defined for an arbitrary partial
order (e.g.\ if the partial order is only made of two incomparable elements).
A generalization of cofinality (resp.\ coinitiality) that is well-defined for
every partial order $P$ is \textdef{set-cofinality} (resp.\
\textdef{set-coinitiality}), namely the least cardinality of a cofinal
(resp.\ coinitial) subset of $P$. The set-cofinality and the set-coinitiality
of $P$ are denoted by $\setcofinality(P)$ and $\setcoinitiality(P)$,
respectively. The two notions clearly agree on linear orders.
%A partial order $P$ is called \textdef{lean} if every subset of cardinality $\ge \setcofinality(P)$ is cofinal in $P$. The notion of lean poset was studied in \cite{DP03}.

In this paper, we study the existence and extendibility of chains in the
Weihrauch degrees. While we are mostly interested in the extendibility of
chains of order type $\kappa$ or $\kappa^*$, for some cardinal $\kappa$ with
$\cofinality(\kappa)>\omega$, many techniques apply to partially ordered
families of degrees as well. We characterize when a chain of Weihrauch degrees
has an upper bound (\thref{thm:no_upper_bound}) and, for every $\eta\le
\continuum := 2^{\omega}$ with uncountable cofinality, we provide an explicit
example of a chain of order type $\eta$ with no upper bound
(\thref{thm:long_chain_no_UB}). This can then be used to show that there are
no cofinal chains in the Weihrauch degrees of any order type
(\thref{thm:no_cofinal_chains_W}).

In contrast, the picture of the coinitiality of the Weihrauch degrees looks
quite different: while there are descending chains of order type $\continuum$
with no non-zero lower bound, the existence of a coinitial sequence in the
Weihrauch degrees is independent of $\ZFC$ and equivalent to $\CH$
(\thref{thm:coinitial_chains_W}). We also characterize when chains of order
type $\kappa+\lambda^*$, for cardinals $\kappa$ and $\lambda$ with
uncountable cofinality, admit an intermediate degree
(\thref{thm:intermediate_degree}), and show that each interval in the
Weihrauch lattice is either finite or uncountable (\thref{thm:intervals_W}).

Finally, we study the extendibility of antichains: we show that every non-trivial
antichain of size $<\continuum$ can be extended (\thref{thm:<c_extendible}),
and provide a necessary condition for an antichain to be maximal
(\thref{thm:extending_antichain_wei_2}).

\subsection{Background}

We now briefly recall the main notions and fix the notation that will be
needed in the rest of the paper. With a small abuse of notation, we will
often identify a Turing/Medvedev/Weihrauch degree with one of its
representatives.

Given two partial multi-valued functions $f,g\pmfunction{\Baire}{\Baire}$, we
say that $f$ is \textdef{Weihrauch reducible} to $g$, and write
$f\weireducible g$, if there are two computable functionals
$\Phi\pfunction{\Baire}{\Baire}$ and
$\Psi\pfunction{\Baire\times\Baire}{\Baire}$ such that, for every
$p\in\dom(f)$,
\[ \Phi(p)\in \dom(g) \land (\forall q\in g(\Phi(p))) (\Psi(p,q)\in f(p)). \]
We use $(\WeiDeg, \weireducible)$ to denote the degree structure of the
\textdef{Weihrauch degrees}, namely, the partial order induced by
$\weireducible$ on the equivalence classes. As a notational convenience, we
also use $(\WeiZero, \weireducible)$ for the restriction of the Weihrauch
degrees to non-empty problems. For a more comprehensive presentation of
Weihrauch reducibility, we refer the reader to \cite{BGP17}. We mention that
while Weihrauch reducibility is often defined in the more general context of
partial multi-valued functions on represented spaces, every Weihrauch degree
has a representative which is a partial multi-valued function on $\Baire$
(see \cite[Lem.\ 11.3.8]{BGP17}). In other words, in order to study the
structure $(\WeiDeg, \weireducible)$, there is no loss of generality in
restricting our attention to problems on the Baire space. With a small abuse
of notation, we can talk about multi-valued functions on spaces like
$\mathbb{N}$, $\baire$, or $(\Baire)^{\mathbb{N}}$, as their elements can be
canonically represented with elements in $\Baire$.

The Weihrauch degrees form a distributive lattice with a bottom element (the
empty set) and no top element. The join and meet operators can be obtained,
respectively, by lifting the following two operators on multi-valued
functions to the degree structure:
\begin{itemize}
	\item $(f_0\sqcup f_1)(i,x):=f_i(x)$, where $\dom(f_0\sqcup f_1):=\bigcup_{i<2} \{i\}\times \dom(f_i)$.
	\item $ (f_0\sqcap f_1)(x_0,x_1):= \bigcup_{i<2} \{i\}\times f_i(x_i)$, where $\dom(f_0\sqcap f_1):=\dom(f_0) \times \dom(f_1)$.
\end{itemize}

We observe that neither of the above two operators can be generalized to
obtain a countable join/meet. Indeed, while the operators
$\sequence{f_i}{i\in\mathbb{N}}\mapsto \bigsqcup_{i\in\mathbb{N}} f_i$ and
$\sequence{f_i}{i\in\mathbb{N}}\mapsto \bigsqcap_{i\in\mathbb{N}} f_i$
defined as $\bigsqcup_{i\in\mathbb{N}} f_i(i,x):=f_i(x)$ and
$\bigsqcap_{i\in\mathbb{N}}
f_i(\sequence{x_j}{j\in\mathbb{N}}):=\bigcup_{i\in\mathbb{N}} \{i\}\times f_i(x_i)$,
respectively, are well-defined and identify an upper and an lower bound for
the family, they do not lift to the Weihrauch degrees and are not, in
general, the supremum or the infimum of $\{f_i \st i\in\mathbb{N}\}$.

As a simple counterexample, let $f_i:=f:=p\mapsto 0$ and observe that $f\weiequiv (p\mapsto 1)$. Now, for every non-computable set $A$, we can consider the family $\family{g_i}{i\in\mathbb{N}}$ defined as $g_i:=p\mapsto 1$ if $i\in A$ and $g_i:=f_i$ otherwise. In other words, we are replacing some of the $f_i$'s with Weihrauch-equivalent problems. However, while $\bigsqcup_{i\in\mathbb{N}} f_i=(i,p)\mapsto 0$, and hence is constant and computable, $\bigsqcup_{i\in\mathbb{N}} g_i$ computes the characteristic function of $A$. This shows that $\bigsqcap_{i\in\mathbb{N}}$ is not a degree-theoretic operator. 

With a similar strategy, we can show that $\bigsqcap_{i\in\mathbb{N}}$ is not degree-theoretic: let $f_i:=f$ be such that $\dom(f):=\{0\}$ and, for a fixed non-computable set $A$, define $g_i$ with $\dom(g_i):=\{1\}$ if $i\in A$ and $g_i:=f$ otherwise. By definition, both $\bigsqcap_{i\in\mathbb{N}} f_i$ and $\bigsqcap_{i\in\mathbb{N}} g_i$ only have one input, but the former has a computable input, while the only input for the latter is the characteristic function of $A$, hence $\bigsqcap_{i\in\mathbb{N}} f_i \not\weireducible \bigsqcap_{i\in\mathbb{N}} g_i$.

It is known that $(\WeiDeg,\weireducible)$ is not an $\omega$-complete join/meet
semilattice (\cite[Cor.\ 3.17]{HiguchiPauly13}). In fact, it is also known
that no non-trivial countable suprema exist.

\begin{theorem}[{\cite[Prop.\ 3.15]{HiguchiPauly13}}]
    \thlabel{thm:no_nontrivial_countable_sup}
    A family $\{ a_i \st i\in\mathbb{N}\}$ of Weihrauch degrees has a supremum iff it is already the supremum of $\{ a_i \st i<n\}$ for some $n\in\mathbb{N}$. In particular, no (strictly ascending) chain $\sequence{a_i}{i\in\omega}$ has a supremum.
\end{theorem}

The dual result for infima does not hold: the infimum of a countable family
does not always exist, but there are chains $\sequence{a_i}{i\in\omega^*}$
with a greatest lower bound (\cite[Example 3.19]{HiguchiPauly13}).

When studying the order-theoretic properties of a reducibility lattice, it is natural to discuss its density. Given two degrees $\mathbf{a} < \mathbf{b}$, we say that \textdef{$\mathbf{b}$ is a minimal cover of $\mathbf{a}$} if there is no degree $\mathbf{c}$ such that $\mathbf{a}<\mathbf{c}<\mathbf{b}$. In other words, $\mathbf{b}$ is a minimal cover of $\mathbf{a}$ if the interval $(\mathbf{a},\mathbf{b})$ is empty. We say that \textdef{$\mathbf{b}$ is a strong minimal cover of $\mathbf{a}$} if, for every $\mathbf{c}<\mathbf{b}$, $\mathbf{c}\le \mathbf{a}$. In other words, the difference between the lower cones of $\mathbf{b}$ and $\mathbf{a}$ is $\{\mathbf{b}\}$. 

Recently \cite{LMPMVMinimalCovers}, the existence and distribution of
(strong) minimal covers in the Weihrauch degrees has been fully
characterized. In particular, it has been shown that the Weihrauch degrees do
not have minimal degrees above $\emptyset$ but are dense (all the
non-degenerate intervals are non-empty) only in the cone above the identity
problem $\id$. Empty intervals exist above every problem $g$ with $\id\not\weireducible g$, while strong minimal covers only exist in the cone below $\id$.

\section{Some results on Medvedev reducibility}

There is a close connection between Weihrauch and Medvedev reducibility, and
some structural results on the Weihrauch lattice can be obtained by analyzing
the properties of the Medvedev degrees. In this section, we briefly recall
the definition and the main properties of Medvedev reducibility. We also
state and prove some results on the Medvedev degrees that, while being
generalizations of already known facts, have not been explicitly observed in
the literature before and are useful to obtain the main results on Weihrauch
reducibility. For a more thorough introduction to Medvedev reducibility, the
reader is referred to \cite{Hinman2012,Sorbi1996}.

Given $A,B\subseteq \Baire$, sometimes referred to as \textdef{mass
problems}, we say that $A$ is \textdef{Medvedev reducible} to $B$, and write
$A\medvedevreducible B$, if there is a computable functional
$\Phi\pfunction{\Baire}{\Baire}$ such that $B \subseteq \dom(\Phi)$ and $\Phi(B)\subseteq A$. The
\textdef{Medvedev degrees} are denoted by $(\MedvedevDeg,
\medvedevreducible)$. They form a distributive lattice with a top element
(the degree of $\emptyset$) and a bottom element (the degree of $\Baire$ or,
equivalently, the degree of any mass problem that contains a computable
point). As for the Weihrauch degrees, we use $(\MedZero, \medvedevreducible)$
for the restriction of the Medvedev degrees to non-empty problems. Since
$\emptyset$ is not the immediate successor of any mass problem, there are no
maximal elements in $\MedZero$. The join and the meet are induced from the
following two operations, respectively:
\begin{itemize}
	\item $A \vee B:= \{ \pairing{p,q} \st p\in A \text{ and } q\in B \} $;
	\item $A \wedge B:= \str{0}\concat A \cup \str{1}\concat B$,
\end{itemize}
where $\pairing{\cdot,\cdot}$ denotes the standard pairing function in the
Baire space and $\str{0}\concat A$ denotes the set obtained by concatenating
the string $\str{0}$ with all the strings in $A$.

Observe that we can restate the definition of Weihrauch reducibility as
follows: $f\weireducible g$ iff there are two computable functionals
$\Phi,\Psi$ such that $\Phi(\dom(f))\subseteq \dom(g)$ and, for every
$p\in\dom(f)$, $\Psi(p,g(\Phi(p))) \subseteq f(p)$. In particular, this shows
that $f\weireducible g$ implies $\dom(g)\medvedevreducible \dom(f)$. This
suggests that the relation between the domains of two multi-valued functions
can be used to define an order-reversing embedding of the Medvedev degrees in the Weihrauch
degrees\footnote{An order-preserving embedding of the Medvedev degrees in the Weihrauch
degrees can be defined by exploiting the range of the computational problems.
For more details, see \cite[Section 5]{BP16}. Infinitely many embeddings can
be defined using the fact that the Weihrauch degrees admit a non-trivial injective
endomorphism \cite[Thm.\ 3.7]{ALMMVJump}.}. Indeed, it is straightforward to
show that $A \medvedevreducible B$ iff $\id_B \weireducible \id_A$, where
$\id_X$ is the restriction of the identity problem $\id$ to the set
$X\subseteq \Baire$.

\begin{remark}
    We highlight that the mapping $d:=A\mapsto \id_A$ induces an isomorphism between $\MedvedevOp=(\MedvedevDeg,\medvedevge)$ and the lower cone of $\id$ in the Weihrauch degrees. This is an important fact that will be frequently used in the rest of the paper.
\end{remark}

The existence and distribution of minimal covers in the Medvedev degrees have
been fully characterized. To describe them, let us define
$\medvedevsucc{p}:=\{\str{e}\concat q \st \Phi_e(q)=p \text{ and }
q\not\turingreducible p \}$.

\begin{theorem}[{\cite[Cor.\ 2.5]{Dyment1976}}]
    For every $A\strictlymedvedevreducible B$, $B$ is a minimal cover of $A$ iff
    \[	(\exists p\in A)(A \medvedevequiv B \wedge \{p\} \text{ and } B
    \wedge \medvedevsucc{p} \medvedevequiv B ).\]
\end{theorem}
Besides, for every $p\in \Baire$, $\{p\}$ is a strong minimal cover of
$\medvedevsucc{p}$ in $\MedvedevOp$. In other words, for every $A\subseteq
\Baire$, $\{p\}\strictlymedvedevreducible A$ implies
$\medvedevsucc{p}\medvedevreducible A$. In particular, since the bottom of
$\MedvedevDeg$ is equivalent to $\{p\}$ for every computable $p$, this
implies that there is a first non-bottom degree in $\MedvedevDeg$, namely,
(the degree of) $\medvedevsucc{0^\mathbb{N}}$. Moreover, being a \emph{degree
of solvability} (i.e., being Medvedev-equivalent to a singleton) is
equivalent to being the top of a strong minimal cover in $\MedvedevOp$, which
proves the first-order definability of the Turing degrees in
$(\MedvedevDeg,\medvedevreducible)$ (\cite[Cor.\ 2.1]{Dyment1976}).

It is known that there are only three mass problems that are
Medvedev-comparable with every other mass problem: these are
$\{0^\mathbb{N}\}$, $\medvedevsucc{0^\mathbb{N}}$, and $\emptyset$. In other
words, for every mass problem $A\notin \{ \{0^\mathbb{N}\},
\medvedevsucc{0^\mathbb{N}}, \emptyset\}$, there is $B\subseteq \Baire$ that
is Medvedev-incomparable with $A$ (\cite[Thm.\ 1.1]{Dyment1976}). The
antichains in $\MedvedevDeg$ can have any cardinality up to $2^\continuum$,
the size of the whole lattice \cite[Thm.\ 4.1]{Sorbi1996}. However, maximal
(non-trivial) antichains need not be infinite.

\begin{proposition}
    \thlabel{thm:max_antichain_M}
    For every $\kappa$ with $1 \le \kappa\le \continuum$, there is a maximal antichain in $\MedvedevDeg$ of size $\kappa$.
\end{proposition}

In the following, we combine the cases of a finite and an infinite
cardinal~$\kappa$ into one, so for convenience, we will use $\sup\kappa$ to
denote~$\kappa-1$ for finite $\kappa>0$, and~$\kappa$ itself for
infinite~$\kappa$.

\begin{proof}
    The case $\kappa=1$ was discussed just before the statement. For $\kappa>1$, we generalize the technique used in \cite[Example 4.2]{Sorbi1996}: let $\family{p_\alpha}{\alpha<\sup \kappa}$ be a non-maximal antichain in the Turing degrees. We define:
    \begin{itemize}
        \item for every $\alpha<\sup \kappa$, $A_{\alpha}:=\{p_\alpha \}$;
        \item $A_{\sup \kappa}:=\{ q \in \Baire \st (\forall \alpha<\sup \kappa)(q\not\turingreducible p_\alpha)\}$.
    \end{itemize}
    It is easy to check that the sequence $\sequence{A_\alpha}{\alpha\le \sup \kappa}$ is an antichain in $\MedvedevDeg$ of size $\kappa$. Indeed, for every $\alpha< \beta < \sup\kappa$, $A_\alpha \medvedevincomparable A_\beta$ follows trivially from the fact that $\family{p_\alpha}{\alpha<\sup \kappa}$ is a Turing-antichain. By definition, $A_{\sup\kappa} \not \medvedevreducible A_{\alpha}$, hence we only need to show that $A_{\alpha}\not\medvedevreducible A_{\sup\kappa}$. Since $\family{p_\alpha}{\alpha<\sup \kappa}$ is not maximal, there is $q$ such that, for every $\alpha<\sup\kappa$, $q \turingincomparable p_\alpha$. Clearly any such $q$ belongs to $A_{\sup\kappa}$, hence $A_{\alpha}\not\medvedevreducible A_{\sup\kappa}$.

    To show that the antichain is maximal, fix $C\subseteq \Baire$ and assume that $(\forall \alpha < \sup \kappa)(C \not\medvedevreducible A_\alpha)$. 
    In particular, for every $q\in C$ and every $\alpha<\sup \kappa$, $q\not\turingreducible p_\alpha$. 
    This implies that $C\subseteq A_{\sup \kappa}$, which in turn implies $A_{\sup \kappa} \medvedevreducible C$.
\end{proof}

\subsection{Chains in the Medvedev degrees}

We now give an overview of some results on the existence of ``long'' chains
in the Medvedev degrees. In \cite{Terwijn08}, it was observed that the
existence of a chain of size $\kappa$ in $(2^\continuum, \subseteq)$ implies
the existence of a chain of the same size in $\MedvedevDeg$ (see the proof of
\cite[Thm.\ 4.2]{Terwijn08}). The converse direction was established in
\cite[Thm.\ 4.3.1]{Shafer2011a}.
% The following result was proved by Shafer in his Ph.D.\ thesis.
% \begin{theorem}[{\cite[Thm.\ 4.3.1]{Shafer2011a}}]
%     \thlabel{thm:long_chains_medvedev}
%     For any cardinal $\kappa$, there is a chain in $(2^\continuum, \subseteq)$ of cardinality $\kappa$ if and only if there is a chain in $\MedvedevDeg$ of cardinality $\kappa$.
% \end{theorem}
Upon closer inspection, the proof of this equivalence can be adapted to
obtain a slightly stronger theorem. We first highlight the following simple
fact:

\begin{proposition}
    \thlabel{thm:chain_upper_bound_M0}
    Let $L$ be a linear order with $\cofinality(L)>\omega$. A chain $\sequence{A_x}{x\in L}$ in $\MedZero$ has an upper bound iff there is a cofinal subsequence $\sequence{A_{x_\alpha}}{\alpha<\cofinality(L)}$ such that $\bigcap_{\alpha<\cofinality(L)} A_{x_\alpha} \neq \emptyset$.
\end{proposition}
\begin{proof}
    Assume first that $B\neq\emptyset$ is an upper bound for $\sequence{A_x}{x\in L}$. Then, there is $e\in\mathbb{N}$ and a cofinal subsequence $\sequence{A_{x_\alpha}}{\alpha<\cofinality(L)}$ such that for every $\alpha<\cofinality(L)$, $A_{x_\alpha}\medvedevreducible B$ via $\Phi_e$. In particular, this implies that
    \[ \emptyset \neq \Phi_e(B)\subseteq \bigcap_{\alpha<\cofinality(L) } A_{x_\alpha}. \]
    Conversely, assume that there is a cofinal subsequence $\sequence{A_{x_\alpha}}{\alpha<\cofinality(L)}$ such that $C:=\bigcap_{\alpha<\cofinality(L) } A_{x_\alpha}\neq \emptyset$. Then, for every $\alpha<\cofinality(L)$, $A_{x_\alpha}\medvedevreducible C$, hence $C$ is an upper bound for $\sequence{A_x}{x\in L}$.
\end{proof}

\begin{theorem}[essentially {\cite[Thm.\ 4.3.1]{Shafer2011a}}]
    \thlabel{thm:long_chains_medvedev}
    For any linear order $L$, the following are equivalent:
    \begin{enumerate}
        \item there is a chain in $(2^\continuum, \supseteq)$ of order type $L$;
        \item there is a chain in $\MedvedevDeg$ of order type $L$;
        \item above every non-top Medvedev degree, there is a chain in $\MedvedevDeg$ of order type $L$.
    \end{enumerate}
    Moreover, if $\cofinality(L)>\omega$, then the existence of a chain of order type $L$ in $(2^\continuum, \supseteq)$ implies the existence of an order-isomorphic chain in $\MedvedevDeg$ with no (non-trivial) upper bound.
\end{theorem}
\begin{proof}
    The implication $(3)\Rightarrow (2)$ is trivial, while $(2)\Rightarrow (1)$ can be readily obtained by inspecting the proof of \cite[Thm.\ 4.3.1]{Shafer2011a}. For $(1)\Rightarrow (3)$, let $\sequence{\mathcal{C}_x}{x\in L}$ be a chain in $(2^\continuum, \supseteq)$. Without loss of generality, we can assume that $\bigcap_{x\in L} \mathcal{C}_x = \emptyset$ (otherwise we can just replace $\mathcal{C}_x$ with $\mathcal{C}_x\setminus \bigcap_{y\in L} \mathcal{C}_y$). Fix a non-empty $A\subseteq \Baire$ and let $p\in A$. Let also $\family{p_\alpha}{\alpha<\continuum}$ be the set of $\turingreducible$-minimal degrees above $p$. Clearly, for every $\alpha$, $A \strictlymedvedevreducible \{p_\alpha\}$. For $x \in L$, let $A_x:= \{ p_\alpha \st \alpha\in\mathcal{C}_x\}$. As in the proof of \cite[Thm.\ 4.3.1]{Shafer2011a}, the family $\{A_x \st x\in L\}$ is a chain in $\MedvedevDeg$ of order type $L$.

    Observe also that $\bigcap_{x\in L} A_x = \emptyset$. If $\cofinality(L)>\omega$ then $\sequence{A_x}{x\in L}$ has no (non-trivial) upper bound. Indeed, for every cofinal subsequence $\sequence{A_{x_\alpha}}{\alpha<\cofinality(L)}$ of $\sequence{A_x}{x\in L}$, $\bigcap_{\alpha<\cofinality(L) } A_{x_\alpha} = \bigcap_{x\in L} A_{x} = \emptyset$, hence the claim follows from \thref{thm:chain_upper_bound_M0}.
\end{proof}

\begin{corollary}
    \thlabel{thm:chain_w1_no_upper_bound_med0}
    There is a chain in $\MedZero$ of order type $\omega_1$ with no upper bound. \qed
\end{corollary}

An explicit example of such sequence can be built as follows: let
$\sequence{d_\alpha}{\alpha<\omega_1}$ be a $\turingreducible$-chain of order
type $\omega_1$. Clearly, this sequence does not have an upper bound, as
lower Turing cones are countable. This also shows that the sequence of mass
problems $\sequence{\{d_\alpha\}}{\alpha<\omega_1}$ has no upper bound in
$\MedZero$, as there are only countably many singletons in the lower Medvedev
cone of every (non-empty) mass problem.

% Embedded in Paul's result
%
% \begin{corollary}
%     For every cardinal $\kappa>\omega$, if there is a chain of size $\kappa$ in $\MedZero$, then there is one with no upper bound.
% \end{corollary}
% \begin{proof}
%     Let $\family{A_\alpha}{\alpha<\kappa}$ be a chain of size $\kappa$ and assume it has an upper bound $B\neq\emptyset$ (otherwise there is nothing to prove). Fix $p\in B$ and let $\sequence{p_\beta}{\beta \in \omega_1}$ be a $\turingreducible$-strictly increasing sequence with $p_0 := p$. The family $\{ A_\alpha \st \alpha <\kappa\} \cup \{ \{p_\beta\} \st \beta < \omega_1\}$ is a chain of size $|\kappa + \omega| = \kappa$ with no upper bound.
% \end{proof}

We highlight that we cannot fully characterize (in $\ZFC$) the cardinals
$\kappa$ for which there is a chain of size $\kappa$ in $\MedvedevDeg$.

\begin{theorem}[{\cite[Cor.\ 4.3.2]{Shafer2011a}}]
    The existence of a chain of cardinality $2^\continuum$ in the Medvedev degrees is independent of $\ZFC$. \qed
\end{theorem}

%Observe that if $\sequence{A_\alpha}{\alpha}$ is a sequence in $\MedZero$ with no upper bound, then $\sequence{id_{A_\alpha}}{\alpha}$ is a sequence in $\WeiDeg$ with no non-zero lower bound. In particular,

% \begin{proposition}
%     \thlabel{thm:A<B}
%     If $A\subset \Baire$ is non-empty then there is a mass problem $B\neq \emptyset$ such that $A\strictlymedvedevreducible B$.
% \end{proposition}

\smallskip

A natural problem is to characterize the cofinality of the Medvedev degrees.
The question is only interesting when considering $\MedZero$, as there is a
top element in $\MedvedevDeg$. We first recall the following well-known fact
about the Turing degrees.

\begin{theorem}[{\cite[Ex.\ V.2.4(c)]{Odifreddi92}}]
    \thlabel{thm:cof_chains_turing}
    The following are equivalent:
    \begin{itemize}
        \item $\CH$;
        \item there is a cofinal chain in $\TuringDeg$;
        \item there is a cofinal chain in $\TuringDeg$ of order type $\omega_1$. \qed
    \end{itemize}
\end{theorem}

Recall also that $\omega_1$-chains in the Turing degrees can be built in
$\ZFC$. However, under $\ZFC+\lnot \CH$, no such chain can be cofinal. We now
show that the analogue of \thref{thm:cof_chains_turing} holds for $\MedZero$.
To this end, we first prove the following lemma.

% \begin{lemma}
%     \thlabel{thm:cof(M_0)>w}
%     No countable family of (non-empty) mass problems is cofinal in $\MedZero$.
% \end{lemma}
% \begin{proof}
%     Observe that, for every family $\family{B_n}{n\in\mathbb{N}}$ of (non-empty) mass problems, we can define $B:=\{ \pairing{p_n}_{n\in \mathbb{N}} \st (\forall n)(p_n \in B_n)\}$. Clearly, for every $n$, $B_n \medvedevreducible B$. Moreover, since $B\neq\emptyset$, there is $C\neq\emptyset$ such that $B\strictlymedvedevreducible C$, which implies that $\family{B_n}{n\in\mathbb{N}}$ is not cofinal.
% \end{proof}

\begin{lemma}
    \thlabel{thm:cofinal_chains_M0_w1}
    If $\sequence{A_\alpha}{\alpha<\kappa}$ is a cofinal chain in $\MedZero$, then $\cofinality(\kappa)=\omega_1$.
\end{lemma}
\begin{proof}
    Recall first of all that for every mass problem $A$, there are $\omega$ many singletons that are Medvedev-reducible to $A$ (as there are only countably many computable functionals). Assume that $\sequence{A_\alpha}{\alpha<\kappa}$ is a cofinal sequence in $\MedZero$ and that $\cofinality(\kappa)>\omega_1$. Let $\sequence{d_\beta}{\beta<\omega_1}$ be a $\turingreducible$-chain. For every $\beta<\omega_1$, there is $\alpha_\beta < \kappa$ such that $\{d_\beta\}\medvedevreducible A_{\alpha_\beta}$. Since $\cofinality(\kappa)>\omega_1$, $\eta:=\sup\{ \alpha_\beta \st \beta < \omega_1\}<\kappa$, and therefore, for every $\beta<\omega_1$, $\{d_\beta\}\medvedevreducible A_{\eta + 1}$, contradicting the fact that there can be only countably many singletons in the lower cone of $A_{\eta + 1}$.

    To conclude the proof, notice that no countable family (and, in particular, no countable chain) can be cofinal in $\MedZero$ (the union of their lower cones does not contain all the singletons). If $\sequence{A_\alpha}{\alpha<\kappa}$ is a cofinal chain and $\cofinality(\kappa)=\omega$, then there is a countable cofinal chain in $\MedZero$, which is a contradiction.
\end{proof}

\begin{theorem}
    \thlabel{thm:cof_chain_med0}
    The following are equivalent:
    \begin{enumerate}
        \item $\CH$;
        \item there is a cofinal chain in $\MedZero$;
        \item there is a cofinal chain in $\MedZero$ of order type
        $\omega_1$.
    \end{enumerate}
\end{theorem}
\begin{proof}
    Let us first show that $(1)\Rightarrow (3)$. By \thref{thm:cof_chains_turing}, there is a cofinal chain $\sequence{d_\alpha}{\alpha<\omega_1}$ in the Turing degrees (identifying a Turing degree with one of its representatives). Define the sequence $\sequence{A_\alpha}{\alpha<\omega_1}$ as $A_\alpha := \{ d_\alpha \}$. Clearly, for every $\alpha<\beta<\omega_1$, $A_\alpha \strictlymedvedevreducible A_\beta$. To prove that $\sequence{A_\alpha}{\alpha<\omega_1}$ is cofinal in $\MedZero$, let $C\subseteq\Baire$ be non-empty and let $p\in C$. Since the sequence $\sequence{d_\alpha}{\alpha<\omega_1}$ is cofinal in $\TuringDeg$, there is $\alpha<\omega_1$ such that $p\turingreducible d_\alpha$, which implies that $C\medvedevreducible A_\alpha$.

    The implication $(3)\Rightarrow (2)$ is trivial, hence we only need to show that $(2)\Rightarrow (1)$. Let $\sequence{B_\alpha}{\alpha<\kappa}$ be a cofinal chain in $\MedZero$. By \thref{thm:cofinal_chains_M0_w1}, we can assume that $\kappa=\omega_1$. For each $\alpha$, choose $p_\alpha \in B_\alpha$. Observe that the set $\{p_\alpha \st \alpha < \omega_1\}$ is cofinal in $\TuringDeg$. Indeed, for every $d\in \Cantor$, there is $\alpha < \omega_1$ such that $\{ d \} \medvedevreducible B_{\alpha}$, hence in particular $d \turingreducible p_\alpha$. We now exploit the set $\{p_\alpha \st \alpha < \omega_1\}$ to define a cofinal sequence $\sequence{q_\alpha}{\alpha<\omega_1}$ in the Turing degrees. By \thref{thm:cof_chains_turing}, this suffices to conclude the proof.

    For every $\alpha$ such that $0<\alpha<\omega_1$, we fix a fundamental sequence for $\alpha$, i.e., a sequence $\sequence{\alpha[n]}{n\in\mathbb{N}}$ such that
    \begin{itemize}
        \item $(\forall n\in\mathbb{N})(\alpha[n]\le \alpha[n+1]< \alpha)$;
        and
        \item $\alpha = \sup \{ \alpha[n]+1 \st n\in\mathbb{N} \}$.
    \end{itemize}
    We then define $q_0:=p_0$ and, for every $\alpha>0$, $q_\alpha := \left( \bigoplus_{n\in\mathbb{N}} q_{\alpha[n]} \right)\oplus p_\alpha$.

    Observe that, if $\alpha < \beta$, then there are $n_0,\hdots, n_k$ such that $\alpha = \beta[n_0][n_1]\hdots [n_k]$, which implies that $q_\alpha \turingreducible q_\beta$. Moreover, for every $\alpha$, $p_\alpha \turingreducible q_\alpha$, hence the sequence $\sequence{q_\alpha}{\alpha<\omega_1}$ is cofinal in $\TuringDeg$, and this concludes the proof.
\end{proof}

%Of course, under $\ZFC+\lnot \CH$ there is no cofinal chain in $\MedZero$. In this case, the chain $\sequence{A_\alpha}{\alpha<\omega_1}$ defined in the proof $(1)\Rightarrow (2)$ above is an explicit example of a $\omega_1$ sequence in $\MedZero$ with no upper bound.

\begin{theorem}
    \thlabel{thm:cofinality_m0}
    $\setcofinality(\MedZero) = \continuum$.
\end{theorem}
\begin{proof}
    Under $\CH$, this is a corollary of \thref{thm:cof_chain_med0} (as no countable family can be cofinal), so assume $\lnot \CH$. Notice first of all that the family of singletons is cofinal in $\MedZero$, hence $\setcofinality(\MedZero)\le \continuum$. Let $\family{A_\alpha}{\alpha<\kappa}$ be a cofinal family of non-empty mass problems with $\kappa \le \continuum$. Let $\varphi$ be the function that maps $p\in \Baire$ to the least $\alpha<\kappa$ such that $\{p\} \medvedevreducible A_{\alpha}$. The cofinality of $\family{A_\alpha}{\alpha<\kappa}$ implies that
    \[ \Baire = \bigcup_{\alpha<\kappa} \varphi^{-1}(\alpha). \]
    Since every non-empty mass problem has countably many singletons in its lower cone, for every $\alpha<\kappa$ the set $\varphi^{-1}(\alpha)$ is countable. This implies $\kappa = \continuum$ because
    \[ \continuum = |\Baire| = \left| \bigcup_{\alpha<\kappa} \varphi^{-1}(\alpha) \right| \le \kappa \cdot \omega = \kappa \le \continuum. \qedhere \]
\end{proof}

\section{Chains in the Weihrauch degrees}
\label{sec:chains}

We now turn our attention to the study of chains in the Weihrauch degrees. As
mentioned, some structural properties of the Weihrauch degrees are immediate
consequences of the results obtained for the Medvedev degrees. For example,
it is observed in \cite[Prop.\ 4.2]{Terwijn08} that under $\ZFC +
2^{<\continuum} = \continuum$, there are chains of size $2^\continuum$ in
$\MedvedevDeg$, and hence in $\WeiDeg$. This follows from
\thref{thm:long_chains_medvedev}, as $2^{<\continuum} = \continuum$ implies
that there are chains of size $2^\continuum$ in $(2^\continuum, \subseteq)$
(\cite[Prop.\ 3.2]{Terwijn08}).

While the following observation is not an immediate consequence of what
happens in the Medvedev degrees, its proof can be obtained by adapting the
technique used in \thref{thm:chain_w1_no_upper_bound_med0}.
\begin{proposition}
    \thlabel{thm:chain-no-upper-bound}
    There is a chain in $\WeiDeg$ of order type $\omega_1$ with no upper bound.
\end{proposition}
\begin{proof}
    Let $\sequence{d_\alpha}{\alpha<\omega_1}$ be a $\turingreducible$-chain. For every $\alpha$, define $f_\alpha\function{\Baire}{\Baire}$ as the constant map $p\mapsto d_\alpha$. Clearly, for every $\alpha < \beta$, $f_\alpha \strictlyweireducible f_\beta$. Assume there is $g$ such that for every $\alpha < \omega_1$, $f_\alpha \weireducible g$. Since there are only countably many computable functionals, there are computable functionals $\Phi,\Psi$ and distinct ordinals $\alpha, \beta$ such that the reductions $f_\alpha \weireducible g$ and $f_\beta \weireducible g$ are witnessed by $\Phi,\Psi$. In particular, for every $p\in g(\Phi(0^\mathbb{N}))$ we have $d_\alpha = \Psi(0^\mathbb{N}, p) = d_\beta$, which is a contradiction.
\end{proof}

It is not hard to see that the previous proof can be adapted to show more
than what is claimed. In the following, we will generalize the previous
proposition and provide a characterization of when a chain of multi-valued
functions admits an upper bound. We first observe the following fact:

\begin{proposition}
    Let $\mathcal{F}$ be an uncountable family of multi-valued functions. If
    \[ (\forall f_0, f_1 \in \mathcal{F}, f_0 \neq f_1)(\exists p\in \dom(f_0) \cap \dom(f_1))(f_0(p)\cap f_1(p)=\emptyset), \]
    then the family has no upper bound, i.e., for every $g$ there is $f\in \mathcal{F}$ such that $f \not \weireducible g$.
\end{proposition}
\begin{proof}
    Assume towards a contradiction that there is $g$ such that for every $f\in \mathcal{F}$, $f\weireducible g$. Since $|\mathcal{F}|>\omega$, there are distinct $f_0, f_1 \in \mathcal{F}$ and computable functionals $\Phi,\Psi$ such that the reductions $f_0\weireducible g$ and $f_1\weireducible g$ are witnessed by $\Phi,\Psi$. By hypothesis, we can fix $p\in \dom(f_0) \cap \dom(f_1)$ such that $f_0(p)\cap f_1(p)=\emptyset$. Clearly, for any solution $q\in g(\Phi(p))$, $\Psi(p,q)$ cannot belong to both $f_0(p)$ and $f_1(p)$, contradicting the definition of Weihrauch reduction.
\end{proof}

% OLD STATEMENT
% just for ordinals/cardinals (proof needs to be fixed anyway, \beta -> \cofinality(\beta))
%
% \begin{lemma}
%     \thlabel{thm:no_upper_bound_1}
%     Let $\beta$ be an uncountable ordinal with $\cofinality(\beta)>\omega$, and let $\sequence{f_\alpha}{\alpha<\beta}$ be a chain of order type $\beta$. If $\sequence{f_\alpha}{\alpha<\beta}$ has an upper bound then there is $E\subset \beta$ with $|E|=|\beta|$ such that for every $p\in\bigcup_{\gamma \in E}\dom(f_\gamma)$, $\bigcap_{\alpha \in I_p^E} f_\alpha(p)\neq \emptyset$, where $I_p^E:=\{\alpha \in E \st p\in\dom(f_\alpha)\}$.
% \end{lemma}
% \begin{proof}
%     Let $g$ be an upper bound for $\sequence{f_\alpha}{\alpha<\beta}$. For every $e,i$, let $E_{e,i}:=\{\alpha < \beta \st f_\alpha \weireducible g \text{ via }\Phi_e, \Phi_i\}$. Since $\bigcup_{e,i\in\mathbb{N}} E_{e,i} = \beta$, there are $e,i\in\mathbb{N}$ such that $|E_{e,i}|=|\beta|$ (this follows from the fact that $\cofinality(\beta)>\omega$). For the sake of readability, let $E:=E_{e,i}$, $\Phi:=\Phi_e$, and $\Psi:=\Phi_i$.

%     Notice that, by definition of Weihrauch reducibility, for every $p\in \bigcup_{\gamma \in E} \dom(f_\gamma)$, $\Phi(p)\in\dom(g)$. Moreover, for every $\alpha \in I^E_p$ and every $q\in g(\Phi(p))$, $\Psi(p,q)\in f_\alpha(p)$. In particular, $\bigcap_{\alpha \in I^E_p} f_\alpha(p) \neq \emptyset$, which concludes the proof.
% \end{proof}

We show that a ``long'' chain (a chain whose order type has uncountable
cofinality) admits an upper bound iff, roughly speaking, there is a cofinal
subchain all of whose problems have common solutions for their inputs
(\thref{thm:no_upper_bound}). The proof is obtained by combining the
following two lemmas.

\begin{lemma}
    \thlabel{thm:no_upper_bound_1}
    Let $L$ be a linear order with $\cofinality(L)>\omega$ and let $\sequence{f_x}{x\in L}$ be a chain of partial multi-valued functions (that is order-isomorphic to $L$). If $\sequence{f_x}{x\in L}$ has an upper bound, then there is $E\subseteq L$ cofinal in $L$ such that, letting $I_p^E:=\{y \in E \st p\in\dom(f_y)\}$,
    \[ \left(\forall p\in\bigcup_{z \in E}\dom(f_z)\right)\left( \bigcap_{y \in I_p^E} f_y(p)\neq \emptyset\right). \]
\end{lemma}
\begin{proof}
    Let $g$ be an upper bound for $\sequence{f_x}{x\in L}$. For every $e,i$, let $E_{e,i}:=\{x \in L \st f_x \weireducible g \text{ via }\Phi_e, \Phi_i\}$. Since $\bigcup_{e,i\in\mathbb{N}} E_{e,i} = L$ and $\omega < \cofinality(L)$, there are $e,i\in\mathbb{N}$ such that $E_{e,i}$ is cofinal in $L$. For the sake of readability, let $E:=E_{e,i}$, $\Phi:=\Phi_e$, and $\Psi:=\Phi_i$.

    Notice that, by definition of Weihrauch reducibility, for every $p\in \bigcup_{z \in E} \dom(f_z)$, $\Phi(p)\in\dom(g)$. Moreover, for every $y \in I^E_p$ and every $q\in g(\Phi(p))$, $\Psi(p,q)\in f_y(p)$. In particular, $\bigcap_{y \in I^E_p} f_y(p) \neq \emptyset$, which concludes the proof.
\end{proof}

If $\cofinality(L) =\omega$ then the conclusion of \thref{thm:no_upper_bound_1} may fail. Given a chain $\sequence{f_n}{n \in \omega}$ such that $\bigcap_{n} \dom(f_n)\neq \emptyset$ (e.g.\ as in \thref{thm:chain-no-upper-bound}), for every $n$, let $h_n$ be defined as $\dom(h_n) := \dom(f_n)$ and $h_n(p) := \{n\}\times f_n(p)$. Clearly $h_n \weiequiv f_n$, but, for any $p \in \bigcap_{n} \dom(f_n)$ and $n\neq m$, $h_n(p) \cap h_m(p) = \emptyset$.
However, the proof of \thref{thm:no_upper_bound_1} can be adapted to obtain that when $L$ is uncountable and $\cofinality(L) =\omega$, for every cardinal $\lambda<|L|$, there is $E\subseteq L$ with $|E|=\lambda$ such that for every $ p\in\bigcup_{z \in E}\dom(f_z)$, $\bigcap_{x \in I_p^E} f_x(p)\neq \emptyset$.

% For the sake of readability, we highlight the following particular case of \thref{thm:no_upper_bound_1}.

% \begin{corollary}
%     Let $\kappa$ be a cardinal with $\cofinality(\kappa)>\omega$ and let $\sequence{f_\alpha}{\alpha<\kappa}$ be a chain of partial multi-valued functions. If $\sequence{f_\alpha}{\alpha<\kappa}$ has an upper bound then there is a cofinal subsequence $\sequence{\overline{f}_\alpha}{\alpha<\cofinality(\kappa)}$ such that, for every $p \in \bigcup_{\alpha<\cofinality(\kappa)} \dom(\overline{f}_\alpha)$,
%     \[ \bigcap \{ \overline{f}_\alpha(p) \st p\in \dom(\overline{f}_\alpha) \} \neq \emptyset. \tag*{$\qed$}\]
% \end{corollary}

We now provide a sufficient condition for a family of
problems $\family{f_x}{x\in P}$, ordered as the partial order $P$, to have an upper bound. We do not need to require that $P$ has uncountable cofinality. However, the statement is trivial if
$\setcofinality(P)=\omega$, as every countable family of problems has an upper bound.

\begin{lemma}
    \thlabel{thm:no_upper_bound_2}
    Let $P$ be a partial order and let $\family{f_x}{x\in P}$ be a family of partial multi-valued functions that is order-isomorphic to $P$. For every $A\subseteq P$ and every $p\in\Baire$, let $I_p^A:=\{x \in A \st p\in\dom(f_x)\}$.

    If there is a cofinal $E\subseteq P$ such that
    \[ \left(\forall p\in \bigcup_{z \in E} \dom(f_z)\right)\left( \bigcap_{y\in I^E_p} f_y(p) \neq \emptyset\right) \]
    then $\family{f_x}{x\in P}$ has an upper bound.
\end{lemma}
\begin{proof}
    Notice first of all that, since $E$ is cofinal in $P$, it is enough to show that $\family{f_z}{z\in E}$ has an upper bound. Let $g\pmfunction{\Baire}{\Baire}$ be the problem with $\dom(g):=\bigcup_{z\in E} \dom(f_z)$ defined as
    \[ g(p):= \bigcap_{y \in I_p^E} f_y(p). \]
    Observe that, for every $z \in E$ and every $p\in\dom(f_z)$, $\emptyset\neq g(p)\subseteq f_z(p)$. This shows that $(\forall z \in E)(f_z \weireducible g)$.
\end{proof}

Combining the previous two lemmas, we obtain the following characterization:

\begin{theorem}
    \thlabel{thm:no_upper_bound}
    Let $L$ be an infinite linear order and let $\sequence{f_x}{x\in L}$ be a chain of partial multi-valued functions. The following are equivalent:
    \begin{enumerate}
        \item $\sequence{f_x}{x\in L}$ has an upper bound in $\WeiDeg$;
        \item $\cofinality(L)=\omega$ or there is a cofinal $E\subseteq L$ such that for every $p\in\bigcup_{z \in E}\dom(f_z)$, $\bigcap_{y \in I_p^E} f_y(p)\neq \emptyset$, where $I_p^E:=\{y \in E \st p\in\dom(f_y)\}$.
    \end{enumerate}
\end{theorem}
\begin{proof}
    \begin{description}
        \item[$(1)\Rightarrow (2)$:] If $\cofinality(L)=\omega$, then this is trivial; otherwise, it follows from \thref{thm:no_upper_bound_1}.

        \item[$(2)\Rightarrow (1)$:] Immediate by \thref{thm:no_upper_bound_2}. \qedhere
    \end{description}
\end{proof}

A natural question is whether \thref{thm:no_upper_bound} can be extended to the case where $P$ is just a partial order and in place of cofinality we consider set-cofinality.
A partial order $P$ is called \textdef{lean} if every subset of cardinality $\ge \setcofinality(P)$ is cofinal in $P$. The notion of lean poset of arbitrary cardinality was studied by Diestel and Pikhurko in \cite{DP03}. It is easy to extend the proof of \thref{thm:no_upper_bound_1}, and hence \thref{thm:no_upper_bound}, to posets with a lean cofinal subset. However Diestel and Pikhurko proved that an infinite poset has a lean cofinal subset if and only if it contains a cofinal chain, so that the generalization to posets with lean cofinal subsets is only apparent.

For the sake of readability, we highlight the following particular case of
\thref{thm:no_upper_bound}.

\begin{corollary}
    \thlabel{thm:no_upper_bound_wo}
    Let $\kappa>\omega$ be a regular cardinal and let $\sequence{f_\alpha}{\alpha<\kappa}$ be a chain of multi-valued functions. The following are equivalent:
    \begin{enumerate}
        \item $\sequence{f_\alpha}{\alpha<\kappa}$ has an upper bound in $\WeiDeg$;
        \item there is a cofinal subsequence $\sequence{\overline{f}_\alpha}{\alpha<\kappa}$ such that for every $p \in \bigcup_{\alpha<\kappa} \dom(\overline{f}_\alpha)$,
        \[ \bigcap \{ \overline{f}_\alpha(p) \st p\in \dom(\overline{f}_\alpha) \}\neq \emptyset. \tag*{\qed}\]
    \end{enumerate}
\end{corollary}

While \thref{thm:no_upper_bound} characterizes precisely when a chain of problems has an upper bound, we only have presented an example of a chain of
order type $\omega_1$ with no upper bound (\thref{thm:chain-no-upper-bound}).
However, explicit examples can be easily obtained for every $\eta \le
\continuum$ with uncountable cofinality. The same strategy can be used to
obtain a chain (with no upper bound) isomorphic to any subchain of
$(2^\continuum, \supseteq)$ with uncountable cofinality.

\begin{theorem}
    \thlabel{thm:long_chain_no_UB}
    For every ordinal $\eta \le \continuum$ with $\cofinality(\eta)>\omega$, there is a chain $\sequence{f_\alpha}{\alpha<\eta}$ in $\WeiDeg$ without upper bound.
\end{theorem}
\begin{proof}
    Let $\family{p_\alpha}{\alpha<\eta}$ be a $\turingreducible$-antichain. For every $\alpha<\eta$, let $f_\alpha$ be the problem with $\dom(f_\alpha):=\{0^\mathbb{N}\}$ defined as $f_\alpha(0^\mathbb{N}):=\{ p_\delta \st \delta \ge \alpha \}$.

    Clearly, for every $\beta<\alpha$, $f_\beta \weireducible f_\alpha$ (as $f_\alpha(0^\mathbb{N}) \subset f_\beta(0^\mathbb{N})$) and $f_\alpha \not\weireducible f_\beta$ (as, by construction, $p_\beta$ does not compute any element in $f_\alpha(0^\mathbb{N})$). Hence, $\sequence{f_\alpha}{\alpha<\eta}$ is a strictly increasing chain in $\WeiDeg$ of order type $\eta$.
    The fact that $\sequence{f_\alpha}{\alpha<\eta}$ does not have an upper bound follows from \thref{thm:no_upper_bound}. Indeed, for every $E\subseteq \eta$ cofinal in $\eta$, $\bigcap_{\alpha\in E} f_\alpha(0^\mathbb{N})=\emptyset$.
\end{proof}
\medskip

We now turn our attention to studying when a sequence of problems admits a
lower bound. Since the Weihrauch degrees have a bottom element, the question
is only interesting when working in $\WeiZero$, i.e., when restricting our
attention to non-trivial lower bounds. Observe that, analogously to what
happens for the upper bounds, every family $\family{f_x}{x\in P}$ with
$\setcoinitiality(P)=\omega$ has a (non-zero) lower bound.

As already mentioned, $\MedvedevOp$ is (isomorphic to) an initial segment of
$\WeiDeg$. As such, examples of chains in $\WeiZero$ with no lower bound can
be readily obtained using \thref{thm:long_chains_medvedev}. In particular, we
highlight the following corollary:

\begin{corollary}
    For every cardinal $\kappa\le \continuum$ with $\cofinality(\kappa)>\omega$, there is a descending sequence $\sequence{f_\alpha}{\alpha< \kappa}$ in $\WeiZero$ with no lower bound.\qed
\end{corollary}

At the same time, it is not hard to show that there are descending sequences
with no lower bound that do not intersect the lower cone of $\id$. As an
explicit example, consider the following: let
$\family{p_\alpha}{\alpha<\omega_1}$ and $q\in \Baire$ be such that
$\sequence{p_\alpha}{\alpha<\omega_1}$ is a $\turingreducible$-chain and, for
every $\alpha$, $q \not\turingreducible p_\alpha$. Define $f_\alpha:=p_\alpha
\mapsto q$. Clearly, $\alpha<\beta$ implies $f_\beta \strictlyweireducible
f_\alpha$. The fact that $\sequence{f_\alpha}{\alpha\in \omega_1^*}$ has no
lower bound in $\WeiZero$ follows from the fact that, if $g\weireducible
f_\alpha$ then $\{p_\alpha\} \medvedevreducible \dom(g)$. However, every mass
problem only has countably many singletons in its lower cone. Finally, for
every $\alpha$, $f_\alpha \not \weireducible \id$ (as $q\not\turingreducible
p_\alpha$).

This example suggests the following simple observation:

\begin{proposition}
    \thlabel{thm:chain_lower_bound}
    Let $\family{f_\alpha}{\alpha<\kappa}$ be a chain in $\WeiZero$ and for each $\alpha$, let $D_\alpha:=\dom(f_\alpha)$. The chain $\family{f_\alpha}{\alpha<\kappa}$ has a lower bound in $\WeiZero$ iff the chain $\family{D_\alpha}{\alpha<\kappa}$ has an upper bound in $\MedZero$.
\end{proposition}
\begin{proof}
    If $g$ is a lower bound for $\family{f_\alpha}{\alpha<\kappa}$, then $\dom(g)$ is a lower bound for $\family{D_\alpha}{\alpha<\kappa}$. On the other hand, if $B$ is an upper bound for $\family{D_\alpha}{\alpha<\kappa}$, then $g:=\id_B$ is a lower bound for $\family{f_\alpha}{\alpha<\kappa}$.
\end{proof}

Despite the fact that a chain with no lower bound in $\WeiZero$ can live
outside the lower cone of $\id$, we now show that, when restricting our
attention to chains of order type $\kappa^*$ for some regular $\kappa$, there
is a precise correspondence between the existence of chains in $\WeiZero$
with no lower bound and the existence of chains in $\MedZero$ with no upper
bound.

\begin{theorem}
    \thlabel{thm:W0_M0}
    For every regular cardinal $\kappa$, the following are equivalent:
    \begin{enumerate}
        \item there is a sequence in $\WeiZero$ of order type $\kappa^*$ which is unbounded below;
        \item there is a sequence in $\MedZero$ of order type $\kappa$ which is unbounded above.
    \end{enumerate}
\end{theorem}
\begin{proof}
The implication $(2)\Rightarrow (1)$ is trivial as the lower cone of $\id$ in $\WeiZero$ is reverse isomorphic to $\MedZero$. 
For the direction $(1)\Rightarrow (2)$, let $\sequence{f_\alpha}{\alpha\in \kappa^*}$ be a descending sequence in $\WeiZero$ with no lower bound. 
In particular, $\kappa>\omega$. 
Let $D_\alpha:=\dom(f_\alpha)$. 
It follows from the definition of Weihrauch reducibility that the family $\family{D_\alpha}{\alpha< \kappa}$ is a chain in $\MedZero$. 
If $\family{D_\alpha}{\alpha< \kappa}$ has order type $\kappa$, then we are done (by \thref{thm:long_chains_medvedev}). 
Assume towards a contradiction that $\family{D_\alpha}{\alpha< \kappa}$ is order-isomorphic to $L$ with $|L|<\kappa$. In other words, let $\sequence{D_x}{x\in L}$ be a subchain of $\family{D_\alpha}{\alpha< \kappa}$ such that for every $\alpha<\kappa$, there is $x\in L$ such that $D_x \medvedevequiv D_\alpha$. 
For every $x\in L$, let $M_x:=\{ \alpha < \kappa \st D_x \medvedevequiv D_\alpha\}$. Since $\kappa$ is regular, there is $x\in L$ such that $|M_x|=\kappa$. 
In particular, since $M_x$ is cofinal in $\kappa$ and $\family{D_\alpha}{\alpha< \kappa}$ is a chain, there is $\alpha_0$ such that for every $\alpha>\alpha_0$, $D_x \medvedevequiv D_\alpha$. 
This implies that $x$ is the top element of $L$, and therefore the sequence $\sequence{D_x}{x\in L}$ has an upper bound in $\MedZero$. 
We have now reached a contradiction, as \thref{thm:chain_lower_bound} implies that the sequence $\sequence{f_\alpha}{\alpha\in \kappa^*}$ has a lower bound in $\WeiZero$.
\end{proof}

\medskip

Having discussed the conditions under which a chain possesses an upper bound
or a lower bound, we now briefly examine when two families of problems admit
an intermediate degree.

\begin{theorem}\thlabel{thm:intermediate_degree}
Let $L,M$ be two linear orders with $\cofinality(L)>\omega$ and $\coinitiality(M)>\omega$. 
Let also $\sequence{f_x}{x\in L}$ and $\sequence{h_z}{z\in M}$ be a two chains in $\WeiDeg$ that are order-isomorphic, respectively, to $L$ and $M$, and such that, for each $x\in L$ and $z\in M$, $f_x\weireducible h_z$. 
Then the following are equivalent:
\begin{enumerate}
\item there is $g$ such that for every $x\in L$ and $z\in M$, $f_x \weireducible g \weireducible h_z$;
\item there are two computable functionals $\Phi,\Psi$ and two sets $X\subseteq L$ and $Z\subseteq M$ cofinal in $L$ and coinitial in $M$ respectively, such that, for every $x\in X$ and $z\in Z$, $f_{x} \weireducible h_{z}$ via $\Phi,\Psi$.
    \end{enumerate}
\end{theorem}
\begin{proof}
For $(1)\Rightarrow (2)$, the set $X$ can be obtained as in the proof of \thref{thm:no_upper_bound_1} using the upper bound $g$ for the family $\sequence{f_x}{x\in L}$. 
In particular, there are $e,i\in\mathbb{N}$ such that for every $x\in X$, $f_x\weireducible g$ via $\Phi_e,\Phi_i$. 
The argument for obtaining $Z$ is symmetrical: since $g$ is a lower bound for $\sequence{h_z}{z\in M}$ and $\coinitiality(M)>\omega$, there are $n,k\in\mathbb{N}$ and a coinitial set $Z\subseteq M$ such that for every $z\in Z$, $g\weireducible h_z$ via $\Phi_n,\Phi_k$.
The maps $\Phi:=\Phi_n\circ\Phi_e$ and $\Psi:=(p,q)\mapsto \Phi_i(p, \Phi_k (\Phi_e(p),q))$ (the compositions of the functionals witnessing the reductions $f_x\weireducible g$ and $g\weireducible h_z$) are the desired functionals.

    To show that $(2)\Rightarrow (1)$, fix $\Phi$, $\Psi$, $X$, and $Z$ as in the hypotheses. We can define $g$ as follows: $\dom(g):=\bigcup_{x\in X} \dom(f_x)$ and $g(p):=\bigcup_{z\in Z} h(\Phi(p))$. Observe that for every $x\in X$, $f_x\weireducible g$ via the maps $\id$ and $\Psi$. Moreover, for every $z\in Z$, $g\weireducible h_z$ via $\Phi$ and $\pi_2:=(p,q)\mapsto q$. The claim follows from the fact that $X$ and $Z$ are cofinal in $L$ and coinitial in $M$ respectively.
\end{proof}

Observe that, unlike what happens when studying upper/lower bounds, the
statement is not trivial if $\cofinality(L)=\omega$ or
$\coinitiality(M)=\omega$. In particular, while there are chains of order
type $\omega+1+\omega^*$ in $\WeiDeg$, it is not clear whether every
$\omega+\omega^*$ chain is extendible to an $\omega+1+\omega^*$ chain.

We mention, however, that no interval in the Weihrauch degrees can have
cardinality $\omega$.

\begin{theorem}
    \thlabel{thm:intervals_W}
    Every infinite interval in $\WeiDeg$ is uncountable.
\end{theorem}
\begin{proof}
    Fix $h\strictlyweireducible f$ and assume that the interval $(h,f)$ is infinite. If the interval $(\dom(f), \dom(h))$ in $\MedvedevDeg$ is infinite, then, by \cite[Thm.\ 2.10]{Terwijn08}, there is an antichain of size $2^\continuum$ between $\dom(f)$ and $\dom(h)$. Fix any such antichain $\family{A_\alpha}{\alpha<2^\continuum}$. We can define $2^\continuum$ many problems $\family{g_\alpha}{\alpha<2^\continuum}$ between $h$ and $f$ as follows: for every $\alpha$, define $g_\alpha\mfunction{\dom(f)\times A_\alpha}{\Baire}$ as $g_\alpha(p,q):=f(p)$. Observe that $\dom(g_\alpha)=\dom(f)\vee A_\alpha \medvedevequiv A_\alpha$, hence $f \not \weireducible g_\alpha$ and $g_\alpha \not\weireducible h$. On the other hand, $g_\alpha \weireducible f$ (trivially) and $h\weireducible g_\alpha$ via the maps $t\mapsto (\Phi(t),\Gamma_\alpha(t))$ and $\Psi$, where $\Phi$ and $\Psi$ are the functionals witnessing $h\weireducible f$ and $\Gamma_\alpha$ witnesses $A_\alpha \medvedevreducible \dom(h)$. Moreover, the family $\family{g_\alpha}{\alpha<2^\continuum}$ is a $\weireducible$-antichain (because their domains are Medvedev-incomparable). This implies that the interval $(h,f)$ has size $2^\continuum$.

    Assume now that the interval $(\dom(f),\dom(h))$ is finite. Assume towards a contradiction that $|(h,f)|=\omega$, and choose a representative $g_n$ for each intermediate degree.

    % For the sake of readability, we say that a countable interval $(h,f)=\{g_n \st n \in \mathbb{N}\}$ is \emph{top-dense} if
    % \[ (\forall n)(\exists m)(g_n \strictlyweireducible g_m \strictlyweireducible f) \tag{$\star$} \]
    \textbf{Claim:} Without loss of generality, we can assume that $f$ is not the minimal cover of any $g_n$.

    Indeed, assume that this is not the case and let $I:=\{ i\in\mathbb{N} \st f \text{ is a minimal cover of }g_i\}$. In particular, for every $i\neq j$, $f\weiequiv g_i \sqcup g_j$ (negating this would lead to a contradiction with the fact that $f$ is a minimal cover of $g_i$). As a consequence of \cite[Thm.\ 1.4]{LMPMVMinimalCovers}, if $f$ is a minimal cover of $g_i$, then $\dom(f)\strictlymedvedevreducible \dom(g_i)$. If $\dom(g_i)\medvedevequiv \dom(g_j)$ then $\dom(f)\medvedevequiv \dom(g_i) \wedge \dom(g_j) \medvedevequiv \dom(g_i)$, against the fact that $\dom(f)\strictlymedvedevreducible \dom(g_i)$. This implies that for every $i\neq j$, $\dom(g_i)\not\medvedevequiv \dom(g_j)$, and therefore $|I|<\omega$ (otherwise, the Medvedev-interval $(\dom(f),\dom(h))$ would be infinite). By the pigeonhole principle, there must be $i\in I$ such that the Weihrauch-interval $(h,g_i)$ is countable. If $g_i$ is not the minimal cover of any problem in the interval $(h,g_i)$, then we are done, as we can replace $f$ with $g_i$ and conclude the proof of the claim. Otherwise, we repeat the same argument with $g_i$ in place of $f$. We observe that this procedure is bound to terminate. Indeed, if not, we are defining a sequence of minimal covers $f >_{\mathrm{W}} g_{i_0} >_{\mathrm{W}} g_{i_1} >_{\mathrm{W}} \hdots$. In particular, since $\dom(f)\strictlymedvedevreducible \dom(g_{i_0})\strictlymedvedevreducible \dom(g_{i_1}) \strictlymedvedevreducible \hdots$, this would imply that the Medvedev-interval $(\dom(f),\dom(h))$ is infinite, which is a contradiction, and hence the claim is proved.

    This implies that the family $\{g_n\st n \in\mathbb{N}\}$ does not have a supremum (otherwise, by \thref{thm:no_nontrivial_countable_sup}, there would be $m\in\mathbb{N}$ such that $g_m = \sup_n \{g_n\st n \in\mathbb{N}\} \strictlyweireducible f$ is a minimal cover). In particular, since $f$ is not the supremum of $\{g_n\st n \in\mathbb{N}\}$, there is $\overline{f} \not \weige f$ such that, for every $n$, $g_n \weireducible \overline{f}$, so $g_n \strictlyweireducible \overline{f}$.

    Consider now the problem $f\sqcap \overline{f}$: clearly, $f\sqcap \overline{f}\strictlyweireducible f$ and for every $n$, $g_n \strictlyweireducible f\sqcap \overline{f}$ (otherwise $g_n$ would be the supremum of $\{g_m\st m \in\mathbb{N}\}$). This contradicts the fact that every degree in the interval $(h,f)$ was represented by some $g_n$, and therefore concludes the proof.
\end{proof}

Observe that examples of finite intervals can be easily obtained as a
corollary of \cite[Cor.\ 2.6]{LMPMVMinimalCovers}: indeed, for every
$n\in\mathbb{N}$, there is a problem $h$ with exactly $n$ minimal covers. The
Boolean algebra obtained by considering the joins of the finitely many
minimal covers of $h$ yields an interval of size $2^n$.

Notice also that the previous proof uses
\thref{thm:no_nontrivial_countable_sup} in a critical way to run a classical
diagonal argument. However, the following result shows that
\thref{thm:no_nontrivial_countable_sup} cannot be extended to
$\continuum$-sized chains, and hence the above strategy does not immediately
yield that no $\continuum$-sized interval exists in the Weihrauch degrees.

\begin{proposition}
    \thlabel{thm:chain_with_sup}
    For every $\kappa\le \continuum$ with $\cofinality(\kappa)>\omega$, there is a chain of order type $\kappa$ in $\WeiDeg$ that admits a supremum.
\end{proposition}
\begin{proof}
    Let $\family{p_\alpha}{\alpha<\kappa +1}$ be a $\kappa$-sized $\turingreducible$-antichain. For every $\alpha\le \kappa$, define $A_\alpha:=\{ p_\gamma \st \gamma < \alpha \}$. Observe that $\family{A_\alpha}{0<\alpha<\kappa}$ is a chain in $\MedZero$ of order type $\kappa^*$ and $A_\kappa = \inf_{\medvedevreducible} \{ A_\alpha \st \alpha <\kappa \}$. Indeed, whenever $0<\alpha<\beta<\kappa$, the reductions $A_\kappa \medvedevreducible A_\beta\medvedevreducible A_\alpha$ are trivial and the separation $A_\alpha \not\medvedevreducible A_\beta$ follows from the fact that $\family{p_\alpha}{\alpha<\eta}$ is an antichain (in particular, $A_\alpha \not \medvedevreducible \{p_\beta\})$. Similarly, $A_\beta \not\medvedevreducible A_\kappa$. To show that $A_\kappa$ is the greatest lower bound, assume that for every $\alpha<\kappa$, $B\medvedevreducible A_\alpha$. Then, since $\cofinality(\kappa)>\omega$, there is $e\in\mathbb{N}$ and a coinitial subsequence $\sequence{\overline{A_\alpha}}{\alpha<\cofinality(\kappa)}$ such that, for every $\alpha$, $B\medvedevreducible \overline{A_\alpha}$ via $\Phi_e$. Since $A_\kappa = \bigcup_{\alpha<\kappa} A_\alpha = \bigcup_{\alpha<\cofinality(\kappa)} \overline{A_\alpha}$, it follows that $B\medvedevreducible A_\kappa$.

    For every $\alpha\le \kappa$, we define $f_\alpha\mfunction{A_\alpha}{\Baire}$ as $f_\alpha(p):=\{p_\delta \st \delta \ge \alpha \}$. It is easy to see that $\family{f_\alpha}{\alpha\le\kappa}$ is a chain of order type $\kappa+1$. Indeed, if $\alpha<\beta\le \kappa$, then $f_\alpha \weireducible f_\beta$ is straightforward (for every $p\in A_\alpha$, $f_\alpha(p)\supset f_\beta(p) \neq\emptyset$), while $f_\beta\not\weireducible f_\alpha$ follows from $A_\alpha\not\medvedevreducible A_\beta$.

    To conclude the proof, let $h$ be an upper bound for $\sequence{f_\alpha}{\alpha<\kappa}$. As in the proof of \thref{thm:no_upper_bound_1}, there are $e,i\in\mathbb{N}$ and a cofinal subsequence $\sequence{\overline{f}_\alpha}{\alpha<\cofinality(\kappa)}$ such that for every $\alpha<\cofinality(\kappa)$, $\overline{f}_\alpha\weireducible h$ via $\Phi_e,\Phi_i$.

    We claim that $f_\kappa \weireducible h$ via $\Phi_e,\Phi_i$. Indeed, for every $p\in \dom(f_\kappa) = \bigcup_{\alpha<\cofinality(\kappa)} \dom(\overline{f}_\alpha)$, $\Phi_e(p)\in \dom(h)$. Moreover, for every $\alpha<\cofinality(\kappa)$ and every $q\in h(\Phi_e(p))$, $\Phi_i(p,q) \in \overline{f}_\alpha(p)$, hence
    \[ \Phi_i(p,q) \in \bigcap_{\alpha<\cofinality(\kappa)} \overline{f}_{\alpha}(p) = \{ p_\delta \st \delta \ge \kappa \} = f_\kappa(p). \qedhere \]
\end{proof}

\subsection{Cofinality and coinitiality}
\label{sec:cofinality_coinitiality}

In this section, we study the cofinality and the coinitiality of the
Weihrauch degrees. We first observe that, while for the Turing and the
Medvedev degrees the existence of a cofinal chain is independent of $\ZFC$
and equivalent to $\CH$, it is provable in $\ZFC$ that there are no cofinal
chains (of any order type) in $\WeiDeg$. To prove this, we first show that
$\setcofinality(\WeiDeg)>\continuum$.

\begin{lemma}
    \thlabel{thm:diag}
    For every family $\mathcal{F}$ of multi-valued functions with $|\mathcal{F}|\le \continuum$ and every infinite $A\subseteq \Baire$ with $|A|\ge |\mathcal{F}|$, there is $g\mfunction{A}{\Baire}$ such that for every $f\in\mathcal{F}$, $g\not\weireducible f$.
\end{lemma}
\begin{proof}
    Let $X\subseteq \Baire$ be such that $|X|= |\mathcal{F}|$ and let $\mathcal{F}=\{ f_x \st x \in X\}$. Let also $\varphi\function{\mathbb{N}\times\mathbb{N}\times X}{A}$ be an injective function.

    Intuitively, we define $g\mfunction{A}{\Baire}$ so that $\varphi(e,i,x)$ is the input for $g$ witnessing the fact that $g\not\weireducible f_x$ via $\Phi_e,\Phi_i$. More precisely, for every $e,i\in\mathbb{N}$ and every $x\in X$, we define $g$ on $p:=\varphi(e,i,x)$ as follows: if $\Phi_e(p) \notin \dom(f_x)$ or there is $q\in f_x(\Phi_e(p))$ such that $(p,q) \notin \dom(\Phi_i)$ then $g(p):=\Baire$. Otherwise, fix $q\in f_x(\Phi_e(p))$ and define $g(p):=\Baire\setminus \{ \Phi_i(p,q) \}$. Finally, define $g(p):=\Baire$ for every $p\notin\ran(\varphi)$.

    By construction, it is immediate to check that for every $x\in X$ and every $e,i \in \mathbb{N}$, $g\not\weireducible f_p$ via $\Phi_e,\Phi_i$ as witnessed by the input $\varphi(e,i,x)$ for $g$.
\end{proof}

As an immediate consequences of \thref{thm:diag}, we obtain:

\begin{corollary}
    \thlabel{thm:cof(W)>c}
    $\setcofinality(\WeiDeg)>\continuum$. \qed
\end{corollary}

\begin{corollary}
    No embedding of $\MedZero$ in the Weihrauch degrees can be cofinal in the Weihrauch degrees.
\end{corollary}
\begin{proof}
    This follows from the previous result and the fact that $\setcofinality(\MedZero)=\continuum$ (\thref{thm:cofinality_m0}).
\end{proof}

\begin{theorem}
    \thlabel{thm:no_cofinal_chains_W}
    There are no cofinal chains in $\WeiDeg$.
\end{theorem}
\begin{proof}
    Observe first of all that every cofinal chain contains a well-ordered cofinal chain, hence, without loss of generality, we can restrict our attention to well-ordered chains.

    Let $\sequence{f_\beta}{\beta< \kappa}$ be a cofinal chain in $\WeiDeg$. Notice that if $\cofinality(\kappa)<\kappa$, then the chain $\sequence{f_\beta}{\beta< \kappa}$ contains a cofinal chain of size $\cofinality(\kappa)$. Assume therefore that $\kappa$ is a regular cardinal. Since $\setcofinality(\WeiDeg)>\continuum$ (\thref{thm:cof(W)>c}), $\kappa >\continuum$. Fix a chain $\sequence{g_\alpha}{\alpha<\omega_1}$ in $\WeiDeg$ with no upper bound (as in  \thref{thm:chain-no-upper-bound} or \thref{thm:long_chain_no_UB}). Since the chain $\sequence{f_\beta}{\beta<\kappa}$ is cofinal in $\WeiDeg$, for every $\alpha<\omega_1$ there is $\beta_\alpha < \kappa $ such that $g_\alpha \weireducible f_{\beta_\alpha}$. Since $\sequence{g_\alpha}{\alpha < \omega_1}$ has no upper bound, the sequence $\sequence{\beta_\alpha}{\alpha<\omega_1}$ must be cofinal in $\kappa$, contradicting the fact that $\kappa=\cofinality(\kappa)> \continuum$.
\end{proof}

Finally, we observe that the coinitiality of the Weihrauch degrees is closely
connected with the cofinality of the Medvedev degrees. We first highlight the
following fact:

\begin{proposition}
    \thlabel{thm:basis}
    There is a continuum-sized family $\family{f_p}{p\in\Baire}$ such that for every $p$, $f_p\weireducible \id$ and for every non-empty $g$, there is $p\in\Baire$ such that $f_p\weireducible g$.
\end{proposition}
\begin{proof}
    For every $p\in \Baire$, we define $f_p\mfunction{\{p\}}{\Baire}$ as $f_p(p):=\Baire$. It is trivial to see that $f_p\weireducible \id$ and, for every non-empty $g$ and for every $p\in \dom(g)$, $f_p\weireducible g$.
\end{proof}

\begin{theorem}
    \thlabel{thm:coinitial_chains_W}
    The set-coinitiality of $\WeiZero$ is $\continuum$. Moreover, the following are equivalent:
    \begin{enumerate}
        \item $\CH$;
        \item there is a coinitial chain in $\WeiZero$;
        \item there is a coinitial chain in $\WeiZero$ of order type
        $\omega_1$.
    \end{enumerate}
\end{theorem}
\begin{proof}
    As a corollary of \thref{thm:basis}, the lower cone of $\id$ is a coinitial subset of $\WeiZero$. Since the lower cone of $\id$ is isomorphic to $\MedvedevOp$, and every coinitial set in $\WeiZero$ must, a fortiori, be coinitial in the lower cone of $\id$, we immediately have $\setcoinitiality(\WeiZero)=\setcofinality(\MedZero)=\continuum$.

    The second part of the statement is a simple consequence of \thref{thm:cof_chain_med0}, as the existence of a coinitial chain in $\WeiZero$ is equivalent to the existence of a cofinal chain in $\MedZero$.
\end{proof}

\section{Antichains}
\label{sec:antichains}

In this section, we prove some results on the size and extendibility of
antichains in the Weihrauch degrees. Clearly, given that the Medvedev degrees
embed as a lattice in the Weihrauch degrees, every antichain in the Medvedev
degrees immediately gives an antichain in the Weihrauch degrees. In
particular, since there are antichains of size $2^\continuum$ in
$\MedvedevDeg$ (\cite{Platek70}), we immediately obtain that the same holds
for $\WeiDeg$ as well. In fact, such antichains can be found ``everywhere''
in the Weihrauch lattice.

\begin{proposition}
    For every problem $f\neq\emptyset$, there is an antichain $\mathcal{A}$ in $\WeiDeg$ of size $2^\continuum$ with $f\in \mathcal{A}$.
\end{proposition}
\begin{proof}
    Fix a problem $f$ and let $p\in \Baire$ be such that $\dom(f)\strictlymedvedevreducible \{p\}$. Let also $\{p_\alpha \st \alpha<\continuum\}$ be the set of minimal degrees above $p$. As in \cite{Platek70} (see also \cite[Thm.\ 4.1]{Sorbi1996}), we can define a $\medvedevreducible$-antichain $\family{A_\beta}{\beta<2^\continuum}$ where $A_\beta \subset \{p_\alpha \st \alpha<\continuum\}$. For every $\beta < 2^\continuum$, define $g_\beta$ as the problem obtained applying \thref{thm:diag} to $\mathcal{F}=\{f\}$ and $A=\{ \str{e,i}\concat p \st e,i\in\mathbb{N} \text{ and } p \in A_\beta \}$. In particular, this guarantees that $g_\beta\not\weireducible f$.

    To conclude the proof, observe that, for every $\beta \neq \gamma$, $g_\beta \weiincomparable g_{\gamma}$ because $\dom(g_\beta) \medvedevequiv A_\beta\medvedevincomparable A_{\gamma} \medvedevequiv \dom(g_{\gamma})$. Similarly, $f \not\weireducible g_\beta$ because $A_\beta \not\medvedevreducible \dom(f)$.
\end{proof}

At the same time, it is trivial to observe that the image of a maximal
antichain in $\MedvedevDeg$ need not be maximal in $\WeiDeg$. In fact, while there are finite maximal antichains in $\MedvedevDeg$ (\thref{thm:max_antichain_M}), 
Dzhafarov, Lerman, Patey, and Solomon proved:

\begin{proposition}[\cite{DamirLuminy}]
    For every countable family $\family{f_n}{n\in\mathbb{N}}$ of non-trivial problems, there is $g$ such that for every $n$, $g \weiincomparable f_n$. \qed
\end{proposition}

In particular, this implies that every countable antichain in $\WeiDeg$ is not maximal, hence the analog of \thref{thm:max_antichain_M} does not hold for the Weihrauch degrees.

However, observe that the previous proposition cannot be extended to
$\continuum$-sized families. This is an immediate consequence of
\thref{thm:basis}. At the same time, the set of problems defined in the proof
of \thref{thm:basis} is far from being an antichain. This is because no
antichain can be coinitial in $\WeiZero$, as the bottom $\emptyset$ is
meet-irreducible (in other words, the meet of any two elements of the
antichain is always a non-empty problem that is not above any element of the
antichain).

We now provide some sufficient conditions that guarantee the extendibility of
an antichain in $\WeiDeg$. For the sake of the presentation, we first state
and prove the following \thref{thm:extending_antichain_wei_1}, while weaker
sufficient conditions will be provided in
\thref{thm:extending_antichain_wei_2}.

\begin{theorem}
    \thlabel{thm:extending_antichain_wei_1} 
    Let $\kappa\le \continuum$. If $\family{f_\alpha}{\alpha<\kappa}$ is an antichain in $\WeiDeg$ of non-trivial problems such that $\family{\dom(f_\alpha)}{\alpha<\kappa}$ is not cofinal in $\MedZero$, then there is $g$ such that for every $\alpha<\kappa$, $g \weiincomparable f_\alpha$.
\end{theorem}
\begin{proof}
    Since the family $\family{\dom(f_\alpha)}{\alpha<\kappa}$ is not cofinal in $\MedZero$, there is a non-empty $D\subset \Baire$ such that for every $\alpha<\kappa$, $D\not \medvedevreducible \dom(f_\alpha)$. Without loss of generality, we can assume that $|D|=\continuum$ (as $D\medvedevequiv D\times \Baire$).

    Thus, if $g$ is a problem with $\dom(g)=D$, then for every $\alpha<\kappa$, $f_\alpha\not\weireducible g$. To conclude the proof, it is enough to define $g$ as the problem obtained applying \thref{thm:diag} to $\mathcal{F}=\family{f_\alpha}{\alpha<\kappa}$ and $A=D$.
\end{proof}

\begin{corollary}
    \thlabel{thm:<c_extendible}
    No antichain $\family{f_\alpha}{\alpha<\kappa}$ in $\WeiDeg$ with $1<\kappa<\continuum$ is maximal.
\end{corollary}
\begin{proof}
    This follows from the fact that no family of size $<\continuum$ is cofinal in $\MedZero$ (\thref{thm:cofinality_m0}). There is a trivial maximal antichain in $\WeiDeg$ of size $1$, namely $\{\emptyset\}$.
\end{proof}

Observe also that no antichain $\mathcal{A}$ can be cofinal in $\MedZero$.
Indeed, since there is no maximal element in $\MedZero$, for every $A\in
\mathcal{A}$ there is a non-empty mass problem $B$ with $A
\strictlymedvedevreducible B$. In particular, since $\mathcal{A}$ is an
antichain, for every $C\in \mathcal{A}$ we have $B \not\medvedevreducible C$.
This simple observation, together with \thref{thm:extending_antichain_wei_1},
immediately yields:

\begin{corollary}
    For every antichain $\mathcal{F}$ in $\WeiDeg$ with $1<|\mathcal{F}|\le \continuum$, if $\{\dom(f) \st f\in \mathcal{F}\}$ is an antichain in $\MedvedevDeg$, then $\mathcal{F}$ is not maximal. \qed
\end{corollary}

We conclude this section by stating and proving the following generalization
of \thref{thm:extending_antichain_wei_1}. In light of \thref{thm:cofinality_m0}, it is enough to consider antichains of size $\continuum$.

\begin{theorem}
    \thlabel{thm:extending_antichain_wei_2}
    Let $\family{f_\alpha}{\alpha<\continuum}$ be an antichain in $\WeiDeg$ of non-trivial problems and let $D_\alpha:=\dom(f_\alpha)$. If the set $\mathcal{B}:=\{ D_\beta \st (\forall x\in D_\beta)(f_\beta(x) \text{ has a } x\text{-computable element})\}$ is not cofinal in $\MedZero$, %$\family{D_\alpha}{\alpha<\continuum}$
    then $\family{f_\alpha}{\alpha<\continuum}$ is not maximal.
\end{theorem}
\begin{proof}
    Assume first of all that $\family{D_\alpha}{\alpha<\continuum}$ is a cofinal set in $\MedZero$ (otherwise, the claim follows by \thref{thm:extending_antichain_wei_1}). Let $\mathcal{C}:=\{ D_\delta \st (\forall \alpha)( D_\alpha \not\medvedevequiv D_\delta \rightarrow D_\alpha \not\medvedevreducible D_\delta) \}$ be the set of $\medvedevreducible$-minimal elements in $\family{D_\alpha}{\alpha<\continuum}$.

    Fix $\gamma$ so that $D_\gamma \notin \mathcal{C}$ and for every $D_\beta\in\mathcal{B}$, $D_\gamma\not\medvedevreducible D_\beta$. Observe that such $\gamma$ exists: since $\family{D_\alpha}{\alpha<\continuum}$ is a cofinal set in $\MedZero$, $\mathcal{C}$ cannot be cofinal in $\family{D_\alpha}{\alpha<\continuum}$, as the $\medvedevreducible$-minimal elements in $\family{D_\alpha}{\alpha<\continuum}$ form a $\medvedevreducible$-antichain, and no antichain can be cofinal in $\MedZero$. This implies that $\mathcal{B}\cup \mathcal{C}$ is not cofinal in $\MedZero$ (in an upper semilattice, the union of two non-cofinal sets is not cofinal), and hence we can find $\gamma$ as above.

    Let $D\medvedevequiv D_\gamma$ be such that $|D|=\continuum$, and fix a bijection $\varphi\function{\mathbb{N}\times D_\gamma\times \continuum \times \mathbb{N}\times \mathbb{N}}{D}$. We define the problem $g\function{D}{\mathbb{N}}$ as follows: fix $p\in D_\gamma$, $\alpha<\continuum$, and $e, i \in \mathbb{N}$. If, for every $n\in\mathbb{N}$, $\Phi_e(\varphi(n,p,\alpha,e,i))\downarrow \in D_\alpha$ and there is $b=b(n)$ such that
    \[ (\forall y \in f_\alpha\circ\Phi_e\circ\varphi(n,p,\alpha,e,i))(\Phi_i(\varphi(n,p,\alpha,e,i), y) = b(n)),\]
    then define $g(\varphi(n,p,\alpha,e,i)):=1-b(n)$. Otherwise, for every $n$, we define $g(\varphi(n,p,\alpha,e,i)):=0$.

    We now show that, for every $\alpha<\continuum$, $g\weiincomparable f_\alpha$. We first show that for every $\alpha$, $f_\alpha \not \weireducible g$. Observe that if $D_\alpha \notin \mathcal{B}$, then $f_\alpha$ has an input $x$ with no $x$-computable solution (while $g$ only has computable solutions), hence $f_\alpha\not\weireducible g$. If $D_\alpha \in \mathcal{B}$, then $f_\alpha\not\weireducible g$ follows immediately from $\dom(g)\medvedevequiv D_\gamma\not\medvedevreducible D_\alpha$.

    It remains to prove that for every $\alpha$, $g\not\weireducible f_\alpha$. Fix $\alpha$ and assume that $g\weireducible f_\alpha$ via $\Phi_e,\Phi_i$. Observe that, if for every $n\in\mathbb{N}$ there is $b=b(n)$ such that for every $y\in f_\alpha\circ \Phi_e\circ \varphi(n,p,\alpha,e,i)$, $\Phi_i(\varphi(n,p,\alpha,e,i),y)=b(n)$, then by construction, $g(\varphi(n,p,\alpha,e,i))\neq b(n)$. On the other hand, let $n\in\mathbb{N}$ and $y_1,y_2\in f_\alpha\circ \Phi_e\circ\varphi(n,p,\alpha,e,i)$ be such that $\Phi_i(\varphi(n,p,\alpha,e,i),y_1)\neq \Phi_i(\varphi(n,p,\alpha,e,i),y_2)$. Since $g$ is single-valued, at least one of the two produced solutions is incorrect, and this concludes the proof.
\end{proof}

\section{Open questions}

In this paper, we explored some questions on the existence of chains and
antichains in the Weihrauch degrees, but many more remain open and require
further investigation. In Section~\ref{sec:chains}, we provided a
characterization of when ``long'' chains are extendible
(\thref{thm:no_upper_bound_wo}). However, \thref{thm:long_chain_no_UB} and
\thref{thm:chain_with_sup} only provide explicit examples of chains of size
$\continuum$. While it follows from the results on the Medvedev degrees that
the existence of a chain of size $2^\continuum$ is consistent with $\ZFC$, a
natural question is the following:

\begin{open}
    Under $\ZFC$, is there a chain of size $2^\continuum$ in $\WeiDeg$? More generally, is every chain of size $\kappa<2^\continuum$ extendible?
\end{open}

Likewise, while we established a connection between chains with no lower
bound in $\WeiZero$ and chains with no upper bound in $\MedZero$
(\thref{thm:W0_M0}), this only applies to well-founded sequences. It is known
that there are no well-founded sequences in $\MedvedevDeg$ of size
$2^\continuum$ (as a consequence of \thref{thm:long_chains_medvedev}). This
suggests the following question:

\begin{open}
    What are the possible order types of chains in $\WeiZero$ with no lower bound? Is there an order type $L$ such that the existence of a chain in $\WeiZero$ of order type $L$ with no lower bound is independent of $\ZFC$?
\end{open}

Another interesting question concerns the chains of order type
$\omega+\omega^*$. \thref{thm:intermediate_degree} characterizes when two
families of problems $\family{f_x}{x\in P}$ and $\family{h_z}{z\in Q}$ have
an intermediate degree (i.e., some $g$ such that, for every $x$ and $z$,
$f_x\weireducible g\weireducible h_z$). However, this only applies to
families with uncountable cofinality/coinitiality.

\begin{open}
    Is there an $\omega+\omega^*$ chain in $\WeiDeg$ with no intermediate degree?
\end{open}

In Section~\ref{sec:cofinality_coinitiality}, we showed that there are no
cofinal chains in $\WeiDeg$ (\thref{thm:no_cofinal_chains_W}) and that the
existence of coinitial chains is equivalent to $\CH$
(\thref{thm:coinitial_chains_W}). While the set-coinitiality of $\WeiZero$ is
$\continuum$, we only obtained a lower bound for its set-cofinality.

\begin{open}
    Is $\setcofinality(\WeiDeg)=2^\continuum$?
\end{open}

In Section~\ref{sec:antichains}, we studied the extendibility of antichains
in $\WeiDeg$ and provided some sufficient conditions under which a
$\continuum$-sized antichain is extendible
(\thref{thm:extending_antichain_wei_2}). A natural question that has proven
challenging to fully resolve is the following:

\begin{open}
    Characterize the maximal antichains in the Weihrauch degrees. Is there a maximal antichain of size $\continuum$?
\end{open}

\bibliographystyle{mbibstyle}
\bibliography{bibliography}

\providecommand{\bysame}{\leavevmode\hbox to3em{\hrulefill}\thinspace}
\providecommand{\MR}{\relax\ifhmode\unskip\space\fi MR }
\providecommand{\MRhref}[2]{%
  \href{http://www.ams.org/mathscinet-getitem?mr=#1}{#2}
}
\providecommand{\href}[2]{#2}
\begin{thebibliography}{10}

\bibitem{ALMMVJump}
Andrews,  Uri, Lempp,  Steffen, Marcone,  Alberto, Miller,  Joseph~S., and
  Valenti,  Manlio, \emph{A jump operator on the Weihrauch degrees},
  Computability \textbf{14} (2025), no.~2, 73--94,
  \doi{10.1177/22113568241309768}. \MR{5000365}

\bibitem{BGP17}
Brattka,  Vasco, Gherardi,  Guido, and Pauly,  Arno, \emph{Weihrauch Complexity
  in Computable Analysis}, Handbook of Computability and Complexity in Analysis
  (Brattka,  Vasco and Hertling,  Peter, eds.), Springer International
  Publishing, Jul 2021, \doi{10.1007/978-3-030-59234-9_11}, pp.~367--417.
  \MR{4300761}

\bibitem{BP16}
Brattka,  Vasco and Pauly,  Arno, \emph{On the algebraic structure of Weihrauch
  degrees}, Logical Methods in Computer Science \textbf{14} (2018), no.~4,
  1--36, \doi{10.23638/LMCS-14(4:4)2018}. \MR{3868998}

\bibitem{DP03}
Diestel,  Reinhard and Pikhurko,  Oleg, \emph{On the cofinality of infinite
  partially ordered sets: factoring a poset into lean essential subsets},
  Order. A Journal on the Theory of Ordered Sets and its Applications
  \textbf{20} (2003), no.~1, 53--66, \doi{10.1023/A:1024449306316}.
  \MR{1993410}

\bibitem{Dyment1976}
Dyment,  Elena~Z., \emph{On Some Properties of the Medvedev Lattice},
  Mathematics of the USSR-Sbornik \textbf{30} (1976), no.~3, 321--340,
  \doi{10.1070/SM1976v030n03ABEH002277}. \MR{0432433}

\bibitem{DamirLuminy}
Dzhafarov,  Damir~D., \emph{Some questions and observations about the structure
  of the Weihrauch degrees}, New directions in computability theory, Luminy,
  France, 2022.

\bibitem{HiguchiPauly13}
Higuchi,  Kojiro and Pauly,  Arno, \emph{The degree structure of Weihrauch
  reducibility}, Logical Methods in Computer Science \textbf{9} (2013),
  no.~2:02, 1--17, \doi{10.2168/LMCS-9(2:02)2013}. \MR{3045629}

\bibitem{Hinman2012}
Hinman,  Peter~G., \emph{A survey of Mu\v{c}nik and Medvedev degrees}, The
  Bulletin of Symbolic Logic \textbf{18} (2012), no.~2, 161--229. \MR{2931672}

\bibitem{LMPMVMinimalCovers}
Lempp,  Steffen, Miller,  Joseph~S., Pauly,  Arno, Soskova,  Mariya~I., and
  Valenti,  Manlio, \emph{Minimal covers in the Weihrauch degrees}, Proceedings
  of the American Mathematical Society \textbf{152} (2024), no.~11, 4893--4901,
  \doi{10.1090/proc/16952}. \MR{4802640}

\bibitem{Odifreddi92}
Odifreddi,  Piergiorgio, \emph{Classical Recursion Theory - The Theory of
  Functions and Sets of Natural Numbers}, 1 ed., Studies in Logic and the
  Foundations of Mathematics, vol. 125, Elsevier, 1992. \MR{982269}

\bibitem{Platek70}
Platek,  Richard~A., \emph{A note on the cardinality of the {M}edvedev
  lattice}, Proceedings of the American Mathematical Society \textbf{25}
  (1970), 917, \doi{10.2307/2036781}. \MR{262080}

\bibitem{Shafer2011a}
Shafer,  Paul, \emph{On the complexity of mathematical problems: Medvedev
  degrees and reverse mathematics}, Ph.D. thesis, Cornell University, 2011,
  p.~195. \MR{2982190}

\bibitem{Sorbi1996}
Sorbi,  Andrea, \emph{The Medvedev Lattice of Degrees of Difficulty},
  Computability, Enumerability, Unsolvability (Cooper,  S.~B., Slaman,  T.~A.,
  and Wainer,  S.~S., eds.), London Math. Soc. Lecture Note Ser., vol. 224,
  Cambridge University Press, Cambridge, New York, NY, USA, 1996,
  \doi{10.1017/CBO9780511629167.015}, pp.~289--312. \MR{1395886}

\bibitem{Terwijn08}
Terwijn,  Sebastiaan~A., \emph{On the Structure of the Medvedev Lattice}, The
  Journal of Symbolic Logic \textbf{73} (2008), no.~2, 543--558,
  \doi{10.2178/jsl/1208359059}. \MR{2414464}

\end{thebibliography}

\end{document}